\documentclass[11pt]{llncs}
\usepackage[a4paper,hmargin=2.5cm,vmargin=2.5cm]{geometry}

\usepackage{amssymb,amsmath,mathtools}
\usepackage{booktabs,enumerate,url}
\usepackage[hidelinks]{hyperref}
\usepackage[title]{appendix}
\usepackage{tikz-cd}
\usetikzlibrary{arrows}

\usepackage[detect-weight]{siunitx}

\newcommand{\Z}{\mathbb Z}
\newcommand{\F}{\mathbb F}
\newcommand{\ea}{\emph{et al.}}

\newcommand{\defeq}{\mathrel{\mathop:}=}
\DeclareMathOperator{\Gal}{Gal}
\DeclareMathOperator{\GL}{GL}

\begin{document}

\title{Computation of a $\num{30750}$-Bit Binary Field Discrete Logarithm}

\author{Robert Granger\inst{1} \and Thorsten Kleinjung\inst{2} \and Arjen K. Lenstra\inst{2} \and\\
  Benjamin Wesolowski\inst{3} \and Jens~Zumbr\"agel\inst{4}}

\institute{Surrey Centre for Cyber Security\\
  Department of Computer Science, University of Surrey, United Kingdom\\
  \email{r.granger@surrey.ac.uk}\\
  \and
  Laboratory for Cryptologic Algorithms,\\
  School of Computer and Communication Sciences, EPFL, Switzerland\\
  \email{thorsten.kleinjung@epfl.ch}\\  
  \and
  Univ. Bordeaux, CNRS, Bordeaux INP, IMB, UMR 5251, F-33400, Talence, France\\
	INRIA, IMB, UMR 5251, F-33400, Talence, France\\
  \email{benjamin.wesolowski@math.u-bordeaux.fr}\\
  \and
  Faculty of Computer Science and Mathematics, University of Passau, Germany
  \email{jens.zumbraegel@uni-passau.de}
}

\maketitle

\begin{abstract}
  This paper reports on the computation of a discrete logarithm in the finite field $\F_{2^{30750}}$, breaking by a large margin the previous record,
  which was set in January 2014 by a computation in~$\F_{2^{9234}}$.
  The present computation made essential use of the elimination step of the quasi-polynomial algorithm due to Granger, Kleinjung and Zumbr\"agel, 
  and is the first large-scale experiment to truly test and successfully demonstrate its potential when applied recursively, which is when it 
  leads to the stated complexity. It required the equivalent of about $\num{2900}$ core years on a single core of an Intel Xeon Ivy Bridge processor running at 
  2.6 GHz, which is comparable to the approximately $\num{3100}$ core years expended for the discrete logarithm record for prime fields, set in a field of 
  bit-length $795$, and demonstrates just how much easier the problem is for this level of computational effort. In order to make the computation feasible 
  we introduced several innovative techniques for the elimination of small degree irreducible elements, 		
  which meant that we avoided performing any costly Gr\"obner basis computations, in contrast to all previous records 	
  since early 2013. While such computations are crucial to the $L(\frac 1 4 + o(1))$ complexity algorithms, they were simply too slow for our purposes. 
  Finally, this computation should serve as a serious deterrent to cryptographers who are still proposing to rely on the discrete logarithm security of 
  such finite fields in applications, despite the existence of two quasi-polynomial algorithms and the prospect of even faster algorithms 
  being developed.

\end{abstract}


\section{Introduction}\label{sec:intro}

Let~$\F_2$ denote the finite field consisting of two elements, let $\F_{2^{30}} \defeq \F_2[T]/(T^{30} + T + 1)$ and let~$t$ be a root of the irreducible
polynomial defining the extension. Furthermore, let $\F_{2^{30750}} \defeq \F_{2^{30}}[X]/(X^{1025} + X + t^3)$ and let~$x$ be a root of the irreducible 
polynomial defining this extension, and consider the presumed generator $g \defeq x + t^9$ of the multiplicative group of $\F_{2^{30750}}$.
To select a target element that cannot be `cooked up' we follow the traditional approach of
using the digits of the mathematical constant~$\pi$, which in the present case leads to
\[
  h_{\pi} \defeq \sum_{i=0}^{30749} \left( \lfloor \pi \cdot 2^{i+1} \rfloor \bmod{2} \right) \cdot t^{29 - (i \bmod{30})} \cdot x^{\lfloor i / 30 \rfloor}
  \in \F_{2^{30750}} .
\]
On 28th May 2019 we completed the computation of the discrete logarithm of $h_{\pi}$ with respect to the base $g$, i.e., an integer solution~$\ell$
to $h_{\pi} = g^{\ell}$. The smallest non-negative solution is given in Appendix~\ref{app:B}, along with a verification script for the computational
algebra system Magma~\cite{magma}. The total running time for this computation was the equivalent of about $\num{2900}$ core years on a 
single core of an Intel Xeon Ivy Bridge or Haswell processor running at 2.6 GHz or 2.5 GHz.
For the sake of comparison, the previous record, which was set in $\F_{2^{9234}}$ in January 2014, required about $45$ core years~\cite{9234Ann}.
Other large scale computations of this type include: the factorisation of a $768$-bit RSA modulus, which took about $\num{1700}$ core years 
and was completed in 2009~\cite{RSA768}; the factorisation of $17$ Mersenne numbers with bit-lengths between $\num{1007}$ and $\num{1199}$
using Coppersmith's `factorisation factory' idea~\cite{CoppersmithFF}, which took about $\num{7500}$ core years and was completed in early 
2015~\cite{MFF}; the computation of a discrete logarithm in a $768$-bit prime field, which took about $\num{5300}$ core years and was 
completed in 2016~\cite{DLP768}; the computation of a discrete logarithm in a $795$-bit prime field, which took about $\num{3100}$ core years, 
and the factorisation of a $795$-bit RSA modulus, which took about $900$ core years, both completed in December 2019~\cite{795Ann}; and
finally, the factorisation of an $829$-bit RSA modulus, which took about $\num{2700}$ core years and was completed in February 2020~\cite{829Ann}.

A natural question that the reader may have is why did we undertake such a large-scale computation? The answer to this question is threefold. 
Firstly, between 1984 and 2013, the fastest algorithm for solving the discrete logarithm problem (DLP) in characteristic two (and more generally in fixed 
characteristic) fields was due to Coppersmith~\cite{coppersmith}. It has heuristic complexity $L_Q(\frac 1 3)$, where for $\alpha \in [0,1]$ and~$Q$ the
cardinality of the field, this notation is defined by
\[ L_Q( \alpha) \defeq \exp \big( O( (\log{Q})^{\alpha} (\log{\log{Q}})^{1-\alpha} ) \big) \]
as $Q \to \infty$; this is the usual complexity measure that interpolates between polynomial time (for $\alpha = 0$) and exponential time
(for $\alpha = 1$), and is said to be subexponential when $0 < \alpha < 1$. In 2013 and 2014 a series of breakthroughs occurred, indicating that the 
DLP in fixed characteristic fields was considerably easier to solve than had previously been believed~\cite{GGMZ13a,Joux13b,GGMZ13b,BGJT13,GKZ14a,GKZ14eprint,JP14,GKZ18}. 
These techniques were demonstrated with the setting of several world records for computations of this type~\cite{1778Ann,1971Ann,4080Ann,6120Ann,6168Ann,4404Ann,9234Ann}, using algorithms that had heuristic complexity at most $L_Q(\frac 1 4 + o(1))$, where $o(1) \to 0$ as $Q \to \infty$.
Amongst all of the breakthroughs from this period, the standout results were two independent and distinct quasi-polynomial algorithms for the DLP 
in fixed characteristic: the first due to Barbulescu, Gaudry, Joux and Thom\'e (BGJT)~\cite{BGJT13}; and the second due to Granger, Kleinjung and 
Zumbr\"agel (GKZ)~\cite{GKZ14eprint,GKZ18}. Both algorithms have heuristic complexity $L_Q(o(1))$, although the latter is rigorous once an 
appropriate field representation is known. As further discussed below, such representations exist, for instance, for Kummer extensions. While 
the computational records demonstrated the practicality of the $L(\frac 1 4 + o(1))$ techniques, it was the promise of the quasi-polynomial
algorithms that effectively killed off the use of small characteristic fields in discrete logarithm-based cryptography, even though these algorithms
had not been seriously tested. In particular, the descent step, or `building block' of the BGJT method has to our knowledge never been used for the 
individual logarithm stage. The descent step of the GKZ method, on the other hand, has been used during the individual logarithm stage for some small
degree eliminations for relatively small fields~\cite{1279Ann,JP14}, but not with a view to properly testing its applicability recursively, as it is
intended to be used. Since the GKZ algorithm becomes practical for far smaller bit-lengths than the BGJT algorithm, one aspect of our motivation was to
test its practicality by seeing in how large a field we could solve discrete logarithms, within a reasonable time. The act of doing so allows one to
assess the practical impact of the algorithm and indeed, its success in this case demonstrates that it can be applied effectively at a large scale.
It also demonstrates just how much easier the DLP is in binary fields than for prime fields, for this level of computational effort.
We note that a recent paper of the second and fourth listed authors proves that the complexity of the DLP in fixed characteristic fields is at most
quasi-polynomial, without restriction on the form of the extension degree -- in contrast to the GKZ algorithm -- thus complementing the practical
impact of the present computation with a rigorous theoretical algorithm for the general case~\cite{KW19}.

Secondly, we wanted to set a significant new record because in our experience of attempting such computations, one almost always encounters many 
unexpected obstacles that necessitate having new insights and developing new techniques in order to overcome them, both of which enrich our knowledge 
and understanding, as well as the state of the art. 
Indeed, it was the solving of a DLP in $\F_{2^{4404}}$ in 2014~\cite{GKZ14a} that led to the discovery of the 
GKZ algorithm. As Knuth has stated, \emph{``The best theory is inspired by practice. The best practice is inspired by theory.''}, which is 
as true in this field as it is in any other. 
In this work, we developed several innovative techniques for the elimination of small degree irreducible polynomials, 
which were absolutely vital to the feasibility of computing a discrete logarithm within such an enormous field. 
In particular, our \emph{new insights and contributions} include the following.
\begin{itemize}
\item A simple theoretical and efficient algorithmic characterisation of all so-called Bluher values.
\item A highly efficient direct method for eliminating degree 2 polynomials over $\F_{q^3}$ on the fly, which applies with probability $\approx \frac{1}{2}$.
\item A highly efficient backup method for eliminating degree 2 polynomials over $\F_{q^3}$ when the direct method fails, making the bottleneck of the computation feasible. In addition we found an interesting theoretical explanation for the observed probabilities.
\item A fast probabilistic elimination of degree 2 polynomials over $\F_{q^k}$ for $k > 3$.
\item A novel use of interpolation to efficiently compute roots of the special polynomials that arise.  
\item An extremely efficient and novel degree 3 elimination method, thus mollifying the new bottleneck in the computation, which was also applied to polynomials of degree 6, 9 and 12.
\item A new technique for eliminating polynomials of degrees 5 and 7, which together with the techniques for the other small degree polynomials meant that \emph{no costly Gr\"obner basis computations were performed} at all. Although such computations are essential to the $L_Q(\frac 1 4 + o(1))$ complexity algorithms, they were simply too slow for our purposes.
\item A highly optimised classical descent using the above elimination costs as input in order to apply a dynamic programming approach.
\end{itemize}

Finally, although discrete logarithm-based cryptography in small characteristic finite fields has effectively been rendered unusable, there are
constructive applications in cryptography for fast discrete logarithm algorithms. For example, knowledge of certain discrete logarithms in binary 
fields can be exploited to create highly efficient, constant time masking functions for tweakable blockciphers~\cite{GJMN16}. Further afield, there 
are several applications which would benefit from efficient discrete logarithm algorithms, although not necessarily in the fixed 
characteristic case. For example, in matrix groups a problem in $\GL(n, q)$ can often be mapped to $\GL(1, q^n)$ and naturally becomes a DLP in 
$\F_{q^n}$.
Also, discrete logarithms are needed in various scenarios when reducing the unit group of an order of a number field modulo a prime. In finite geometry 
one encounters near-fields, where multiplication usually requires computing discrete logarithms, and there are other applications 
in representation theory, group theory and Lie algebras in the modular case. The relevance of the DLP to computational mathematics therefore extends 
far beyond cryptography. Despite the existence of the quasi-polynomial algorithms, the use of prime degree extensions of $\F_2$ has recently been 
proposed in the design of a secure compressed encryption scheme of Canteaut~\ea~\cite[Sec.~5]{Canteaut}. 
For $80$ bits of security a prime extension of degree $\approx \num{16000}$ was suggested (and of degree 
$\num{4000000}$ for $128$-bit security), very conservatively based on the polynomial time first 
stage of index calculus as described in~\cite{JP14}, but ignoring the dominating quasi-polynomial individual logarithm stage. Even though the security of 
such fields is massively underestimated and the resulting schemes are prohibitively slow, it is interesting, to say the least, that these proposals were 
made. Although the computation reported here is not for a prime degree extension of $\F_2$, our result should nevertheless be regarded as a serious 
deterrent against such applications. Indeed, the central remaining open problem in this area, which is to find a polynomial time algorithm for the DLP in fixed characteristic fields (either rigorous or heuristic) is far more likely to be solved now than it was before 2013.
 
Another natural question that the reader may have is why $\num{30750}$, exactly? 
In contrast to the selection of $h_{\pi}$, this extension degree was most 
certainly `cooked up' specifically for setting a new record\footnote{The fact that $\num{30750}$ is precisely the seating 
capacity of the first listed author's home football stadium is purely a coincidence; see~\url{https://en.wikipedia.org/wiki/Falmer_Stadium}.}. In
particular, $\num{30750} = 3 \cdot 10 \cdot (2^{10} + 1)$ so that the field is of the form $\F_{q^{3(q+1)}}$, thus permitting the use 
of a twisted Kummer extension, in which it is much easier to solve logarithms than in other extensions of the same order of magnitude. Firstly, as explained in~\S\ref{sec:degree1} the factor base which consists of all monic degree one elements possesses an automorphism of order 
$\num{3075}$. This reduces the cardinality of the factor base from $2^{30}$ to a far more manageable $\num{349185}$. Secondly, such field representations ensure that when eliminating an element during the descent step, the cofactors have minimal degree and therefore will be smooth with maximum probability (under a uniformity assumption when not using the GKZ elimination step). Thirdly, there are no so-called `traps' during the descent, i.e., elements which can not be eliminated~\cite{GKZ18}. The first two reasons explain why Kummer and twisted Kummer extensions have been used for all of the records since 2013.
Observe that for the previous record we have $\num{9234} = 2 \cdot 9 \cdot (2^9+1)$. However, whereas for the $\num{9234}$-bit computation the 
factor base consisted of all degree one and irreducible degree two elements, for the present computation it would be too costly to compute and store the 
logarithms of such a factor base, even when incorporating the automorphism. So it was essential that the logarithms of degree two elements could be 
computed `on the fly', i.e., as and when needed, without batching. For a base field of the form $\F_{q^3}$ this is non-trivial as each degree two element 
only has a half chance of being eliminable when the method of~\cite{GGMZ13a} is applied. For our target field, a backup approach similar in spirit to that 
used for the $\num{6120} = 3 \cdot 8 \cdot (2^8-1)$-bit record can be used~\cite{6120Ann,GGMZ13b}, and the chosen parameters allow for this with an acceptable probability. However, in order to make the computation feasible, as already mentioned it was essential to develop new techniques for the elimination of degree two and other small degree polynomials and to optimise them algorithmically, arithmetically and implementation-wise. For the number of core years used we therefore believe with high confidence that we have solved as large a DLP as is possible with current state of the art techniques.

The sequel is organised as follows. In~\S\ref{sec:Bluher} we present some results on so-called Bluher polynomials that are essential to the theory and
practice of our computation. In~\S\ref{sec:GKZ} we recall the GKZ algorithm, explaining the importance of degree two elimination and how this 
leads to a quasi-polynomial algorithm. We describe the field setup in~\S\ref{sec:setup}, and in~\S\ref{sec:degree1} describe how we computed the 
logarithms of the factor base elements, namely the degree one elements. Subsequently we present our degree two elimination methods 
in~\S\ref{sec:degree2} and more generally even degree element elimination in~\S\ref{sec:smalleven}, and in~\S\ref{sec:smallodd} detail how we 
eliminated small odd degree elements. In~\S\ref{sec:BluherRoots} we explain how we efficiently compute the roots of Bluher polynomials in various
scenarios, while in~\S\ref{sec:classical} we detail the classical descent strategy, analysis and timings. In~\S\ref{sec:canteaut} we briefly discuss the security of the proposal of Canteaut~\ea, and finally in~\S\ref{sec:conclusion} we make some concluding remarks.


\section{Bluher Polynomials and Values}\label{sec:Bluher}

Let $q$ be a prime power, let $k \ge 3$ and consider the polynomial $F_B(X) \in \F_{q^k}[X]$  defined by
\begin{equation}\label{eqn:bluherpoly}
F_B(X) \defeq X^{q+1} - BX + B.
\end{equation}
We call such a polynomial a Bluher polynomial. A basic question that arises repeatedly for us is for what values $B \in \smash{\F_{q^k}^{\times}}$ does 
$F_B(X)$ split over $\F_{q^k}$, i.e., factor into a product of $q+1$ linear polynomials in $\F_{q^k}[X]$? With this in mind we define $\mathcal{B}_k$ to 
be the set of all $B \in \F_{q^k}^{\times}$ such that $F_B(X)$ splits over $\F_{q^k}$, and call the members of $\mathcal{B}_k$ Bluher values. It turns 
out that there are some simple characterisations of $\mathcal{B}_k$ which enable various computations to be executed very efficiently. We first recall 
a result of Bluher.
\begin{theorem}{\cite{bluher}}\label{thm:BluherCount}
  The number of elements $B \in \F_{q^k}^{\times}$ such that the
  polynomial $F_B(X)$ splits completely over $\F_{q^k}$ equals
  \[ \dfrac{q^{k-1}-1}{q^{2} - 1} \quad \text{if~$k$ odd,} \qquad
  \dfrac{q^{k-1}-q}{q^{2} - 1} \quad \text{if~$k$ even.} \]
\end{theorem}
One characterisation of $\mathcal{B}_k$ is the following generalisation of a theorem due to Helleseth and Kholosha.
\begin{theorem}{\cite[Lemma.~4.1]{GKZ18}}\label{thm:HellesethKholosha}
We have
\[ \mathcal{B}_k = \Big\{ \frac{(u - u^{q^2})^{q+1}} {(u - u^{q})^{q^2 + 1}}
  \,\big\vert\, u \in \F_{q^k} \setminus \F_{q^2} \Big\} \,. \]
\end{theorem}

Theorem~\ref{thm:HellesethKholosha} is useful for sampling from $\mathcal{B}_k$. However, sometimes it is desirable to test whether a given element of 
$\F_{q^k}$ is a Bluher value, without factorising (\ref{eqn:bluherpoly}) and without first enumerating~$\mathcal{B}_k$ and checking membership, which 
may be intractable when~$k$ is so large that there are prohibitively many Bluher values to precompute. To this end, we define the following polynomials.
Let $P_1(X) = 1$, $P_2(X) = 1$, and for $i \ge 3$ define $P_i(X)$ by the recurrence
\begin{equation}\label{eqn:BluherRecurrence}
P_i(X) = P_{i-1}(X)  - X^{q^{i-3}} P_{i-2}(X) . 
\end{equation}

\begin{theorem}\label{thm:BluherRecurrence}
An element $B \in \F_{q^k}^{\times}$ is a Bluher value if and only if $P_k(\frac 1 B) = 0$.
\end{theorem}

\begin{proof}
  A simple induction shows that the degree of $P_k$ equals the number of
  Bluher values given in Theorem~\ref{thm:BluherCount}, so it suffices to
  prove the only if part.
  Let $B$ be a Bluher value and $C = \frac1B$.
  Using $X^q \equiv \frac{X - 1}{CX} \pmod{F_B}$ and induction, for $i \ge 2$
  one obtains
  \[ X^{q^i} \equiv \frac{P_{i+1}(C)X-P_i^q(C)}{C^{q^{i-1}} (P_i(C)X-P_{i-1}^q(C))}
  \pmod{F_B}. \]
  Since $X \equiv X^{q^k} \pmod{F_B}$ by assumption,
  $C^{q^{k-1}} (P_k(C)X-P_{k-1}^q(C)) X \equiv P_{k+1}(C)X-P_k^q(C) \pmod{F_B}$
  and thus $P_k^q(C)=0$.
\end{proof}

One can efficiently test if an element $B \in \smash{\F_{q^k}^{\times}}$ is a Bluher value simply by evaluating $P_k(\frac 1 B)$ using the
recurrence~(\ref{eqn:BluherRecurrence}). By rewriting the recurrence in a matrix form it is possible to apply fast methods for evaluating it.
However, since~$k$ is quite small for most of the computations, we did not use this. 
Furthermore, one can obviously compute $\mathcal{B}_k$ by applying the recurrence~(\ref{eqn:BluherRecurrence}) and factorising $P_k$.


\section{The GKZ Algorithm}\label{sec:GKZ}

In this section we sketch the main idea behind the GKZ algorithm, referring the reader to the original paper for rigorous statements and 
proofs~\cite{GKZ18}. 

Assume there exist coprime $h_0, h_1 \in \F_{q^k}[X]$ of degree at most two such that there exists a monic irreducible 
polynomial~$I$ of degree~$n$ with $I \mid h_1 X^q - h_0$.  Let $\F_{q^{kn}} \defeq \F_{q^k}[X]/(I) = \F_{q^k}(x)$ where~$x$
is a root of~$I$.  Then we have the following commutative diagram:
\begin{center}\begin{tikzcd}
& \F_{q^k}[X,Y] \arrow[ld, "Y \mapsto X^q"']\arrow[rd, "Y \mapsto h_0/h_1"] &\\
R_1 = \F_{q^k}[X]\arrow[rd] & & R_2 = \F_{q^k}[X][1/h_1]\arrow[ld]\\
& \F_{q^k}[X]/(I). &
\end{tikzcd}\end{center}
Assume further for the moment that for any $d \ge 1$, it is possible in time polynomial in~$q$ and~$d$ to express a given irreducible degree
two element $Q \in \F_{q^{kd}}[X]$ as a product of at most $q+2$ linear elements in $\F_{q^{kd}}[X]$ mod $I$, i.e.,
\begin{equation}\label{eq:degree2}
Q \equiv \prod_{i=1}^{q+2} (X + a_i) \pmod{I}, \quad \text{with} \ a_i \in \F_{q^{kd}} .
\end{equation}
We have the following.
\begin{proposition}{\cite[Prop.~3.2]{GKZ18}}\label{prop:buildingblock}
For the stated field representation, let $d \ge 1$ and let $Q \in \F_{q^{k}}[X]$ be an irreducible polynomial of degree $2d$. Then $Q$ can be 
rewritten in terms of at most $q + 2$ irreducible polynomials of degrees dividing $d$ in an expected running time polynomial in~$q$ and in~$d$.
\end{proposition}

We now sketch the proof. Over $\F_{q^{kd}}$, we have the factorisation $Q = \prod_{i=1}^{d} Q_i$, where $Q_i \in \F_{q^{kd}}[X]$
are irreducible quadratics. Applying the above rewriting assumption to any one of the elements~$Q_i$ gives an instance of the product 
in~(\ref{eq:degree2}), with $Q_i$ on the l.h.s.  Taking the norm map, i.e., the product of all conjugates of each term under $\text{Gal}(\F_{q^{kd}}/\F_{q^{k}})$,
produces the original $Q$ on the l.h.s.\ and for each term on the r.h.s., a $d_1$-th power of an irreducible polynomial 
in $\F_{q^k}[X]$ of degree $d_2$, where $d_1d_2 = d$, which is what the proposition states.

Now, if one begins with an irreducible polynomial $Q$ of degree $2^e$ with $e \ge 1$, then recursively applying Proposition~\ref{prop:buildingblock} allows
one to express $Q$ as a product of at most $(q+2)^e$ linear polynomials. Moreover, given $g \in \F_{q^{kn}}^{\times}$ and $h \in \langle g \rangle$ 
an element whose logarithm with respect to $g$ is to be computed, one can efficiently find an irreducible representative of $h$ mod $I$ of degree 
$2^e$, provided that $2^e > 4n$, by adding random multiples of $I$ to $h$ so that the degree of $h + rI$ is $2^e$~\cite[Lem.~3.3]{GKZ18}.

The logarithms of the linear elements -- the factor base -- remain to be computed, but this can be accomplished in polynomial time in various 
ways~\cite{GGMZ13a,Joux13b}, or can be obviated altogether by repeating the descent to linear elements as many times as the cardinality of the factor 
base~\cite{EngeGaudry}, but the latter is usually only of theoretical interest. Note that in practice the above descent method is not optimal as for 
high degree element elimination the classical methods are more efficient, since they produce far fewer descendants.

The crux of this approach, namely eliminating a degree two element as expressed in~(\ref{eq:degree2}), can be achieved in various ways originating 
with~\cite{GGMZ13a}, as follows. Let $a, b, c \in \F_{q^k}$ and consider the polynomial $X^{q+1} + a X^q + b X + c$.
Firstly, since $X^q \equiv h_0/h_1 \pmod I$ we have
\begin{equation}\label{eqn:basicrewrite}
X^{q+1} + a X^q + b X + c \equiv \tfrac 1 {h_1} \big( (X + a) h_0 + (b X + c) h_1 \big) \!\pmod I,
\end{equation}
with the numerator of the r.h.s.\ of~(\ref{eqn:basicrewrite}) of degree at most three. We would like to impose that this numerator is divisible by $Q$. Therefore, let $L_Q \in \F_{q^k}[X]^2$ be the lattice defined by
\[ L_Q \defeq \{( w_0, w_1) \in \F_{q^k}[X]^2 \mid w_0 h_0 + w_1 h_1 \equiv 0 \!\!\pmod{Q} \} . \]
In general, $L_Q$ has a basis of the form $(1,u_0X + u_1), (X,v_0X +v_1)$ with $u_i,v_i \in \F_{q^k}$ and thus in order for the r.h.s.\
of~(\ref{eqn:basicrewrite}) to be divisible by $Q$ we must choose $a,b,c$ such that 
\begin{equation}\label{eqn:abc}
(X + a, b X + c) = a (1, u_0 X + u_1) + (X, v_0 X + v_1) .
\end{equation}
When this is so, the cofactor of $Q$ has degree at most one.

Secondly, provided that $b \ne a^q$ and $c \ne a b$, the polynomial $X^{q+1} + a X^q + b X + c$ may be transformed by using the substitution
\begin{equation}\label{eqn:subst}
X \longmapsto \frac {a b - c} {b - a^q} X - a
\end{equation}
into a scalar multiple of the Bluher polynomial $X^{q+1} - B X + B$, where 
\begin{equation}\label{eqn:Bluhertransform}
B \defeq \frac {(b - a^q)^{q+1}} {(c - a b)^q}.
\end{equation}
Thus, if these two conditions on $a,b,c$ hold then $X^{q+1} + a X^q + b X + c$ splits whenever $B \in \mathcal{B}_k$. Combining~(\ref{eqn:abc})
and~(\ref{eqn:Bluhertransform}) yields the condition
\begin{equation}\label{Bluhermaineqn}
B = \frac{(-a^q + u_0 a + v_0)^{q+1}} {(-u_0 a^2 + (u_1 - v_0)a + v_1)^q}.
\end{equation}
In order to eliminate $Q$, one chooses $B \in \mathcal{B}_k$ to obtain a univariate polynomial in $a$, which can be checked for roots in $\smash{\F_{q^k}}$ 
in various ways, for instance by computing the GCD with $a^{q^k} - a$. Note that the more Bluher values there are, the higher the chance is of 
eliminating $Q$. By Theorem~\ref{thm:BluherCount} the hardest case is for $k = 3$ as there is only one such Bluher value, namely $1$, which is the 
case we address in this paper, using a backup elimination method \`a la~\cite{GGMZ13b}. By this rationale, in general when 
Proposition~\ref{prop:buildingblock} is applied, the elimination of degree two $Q_i$ over $\F_{q^{kd}}$ becomes `easier' as $d$ grows, although
the computational cost per Bluher value increases as the base field increases in size.


\section{Field Setup and Factors}\label{sec:setup}

Recall that the group in which we look at the discrete logarithm
problem is the (cyclic) unit group of the finite field
$\F_{2^{30750}}$, thus its group order is $N \defeq 2^{30750} - 1$.
We used the factorisation algorithm of Magma, which is assisted by the
`Cunningham tables' for Mersenne numbers, to obtain the~$58$ known
prime factors of~$N$ up to bit-length~$135$ (see Appendix~\ref{app:A}).%
\footnote{We remark that this partial factorisation is sufficient,
  since the expected running time of a generic discrete logarithm
  algorithm on a subgroup of order the largest listed prime already
  exceeds our total running time.}

Also recall that to represent the field $\F_{2^{30750}}$, we first let
$\F_{2^{30}} \defeq \F_2[t] = \F_2[T] / (T^{30} + T + 1)$ where
$t \defeq [T] \in \F_{2^{30}}$, and then, setting $\gamma \defeq t^3$,
we define the target field as the extension \[ \F_{2^{30750}} \,\defeq\,
  \F_{2^{30}}[x] \,=\, \F_{2^{30}}[X] / (X^{1025} + X + \gamma) \]
where $x \defeq [X] \in \F_{2^{30750}}$.%
\footnote{Note that this convenient field representation for the
  discrete logarithm problem can be assumed without loss of
  generality, as it is possible to efficiently map between arbitrary
  field representations given by irreducible polynomials,
  cf.~\cite{Lenstra1991}.}
The field element $g \defeq x + t^9$ is then a supposed generator
for the group, as $g^{N / p} \ne 1$ for the 58 listed prime factors
$p \mid N$.

The schedule and the running times of our computation of the discrete
logarithm of the target element $h_{\pi} \in \F_{2^{30750}}$ from the
introduction are summarised in Table~\ref{tab:rt}.  In the subsequent
sections we present the corresponding steps in more detail.

\begin{table}
  \begin{center}
    \caption{Overall schedule and running times.} \label{tab:rt}
    \begin{tabular}{lcr}
      Step & when & core hours \\\hline
      Relation Generation & 23 Feb 2016 & 1 \\
      Linear Algebra & 27 Feb 2016 - 24 Apr 2016 & \num{32498} \\
      Initial Split & 09 Sep 2016 - 20 Sep 2016 & \num{248140} \\
      Classical Descent & 13 Oct 2016 - 23 Jan 2019 & ~\num{17077836} \\
      Small Degree Descent~ & 27 Jan 2019 - 28 May 2019 & \num{8122744} \\\hline
      \emph{total running time} & & \num{25481219}
    \end{tabular}
  \end{center}
\end{table}


\section{Logarithms of Degree 1 Elements}\label{sec:degree1}

As usual this consists of relation generation followed by a linear algebra elimination.

\subsection{Relation Generation}

The relation generation step is based on the following.  As before,
let~$\F_q$ be a finite field, let $k \ge 3$ and
let $a, b, c \in \smash{\F_{q^k}}$.  Then over~$\smash{\F_{q^k}}$, the
polynomial $X^{q+1} + a X^q + b X + c$ splits if and only if its
affine transformation $X^{q+1} - B X + B$ splits, with~$B$ as given
in~(\ref{eqn:Bluhertransform}).  According to
Theorem~\ref{thm:BluherCount}, for random $a, b, c$ this occurs with
probability about $q^{-3}$, in which case the set of roots of
$X^{q+1} + a X^q + b X + c$ is, by the
transformation~(\ref{eqn:subst}), given by
\[ \big\{ \lambda z - a \mid z \in \F_{q^k} \text{ a root of
    $X^{q+1} - B X + B$} \big\}, \quad \text{where} \quad \lambda
  \defeq \frac {a b - c} {b - a^q}. \]
In our case we have $\smash{\F_{q^k} = \F_{2^{30}}}$ and regard all
linear polynomials $x + u$, for $u \in \F_{2^{30}}$, as the
factor base.  Since in the target field one has
\[ x^{1024} \,=\, \frac {x + \gamma} x, \]
one could use $q = 2^{10}$ and $k = 3$ for a fast relation generation.
However, in order to decrease the row-weight of the resulting matrix,
we prefer to choose $\tilde k = 6$ and $\tilde q = 2^5$, so that
$x^{\tilde q^2} = \frac {x + \gamma} x$.  The relation generation
then makes use of the identity
\[ (x^{\tilde q+1} + a x^{\tilde q} + b x + c)^{\tilde q}
  = \tfrac 1 x \big( (b^{\tilde q} + 1) x^{\tilde q+1} + \gamma x^{\tilde q}
  + (a^{\tilde q} + c^{\tilde q}) x + a^{\tilde q} \gamma \big) , \]
which provides a relation of weight~$66$ each time both sides split.

Letting again $q = 2^{10}$, the Galois group $\Gal(\F_{2^{30750}} / \F_q)$
of order~$\num{3075}$ acts on the factor base elements.  Indeed, we have
\[ (x + u)^q \,=\, \tfrac {u^q + 1} x \big( x + \tfrac {\gamma}
  {u^q + 1} \big), \] by which we can express all logarithms in
a Galois orbit, up to a small cofactor, in terms of the logarithms
of a single representative (and of~$x$).  This way we only need to
consider the $\num{349185}$ orbits as variables.  Within a running time of
$63$ minutes on an ordinary desktop computer the relation generation
was finished.

\subsection{Linear Algebra}

The relation generation produced a $\num{349195} \times \num{349184}$
matrix with~$66$ nonzero entries in each row, each being of the form
$\pm 2^e \in \Z / N \Z$ with some $e \in \Z / 30750 \Z$.  We solved
the linear algebra system for this matrix using the Lanczos
method~\cite{Lanczos,odlyzko}.  More precisely, with the intention to
speed up modular computations, we used the factorisation
\[ N \,=\, N_1 \cdot N_2 \cdot N_3 \,\defeq\, (2^{10250} - 1) \cdot
  (2^{10250} + 2^{5025} + 1) \cdot (2^{10250} - 2^{5025} + 1) \] and
computed a solution vector for each of three moduli $M_i \mid N_i$
having small cofactors.  This took~$\num{30393}$ core hours on single
computing nodes having~$16$ or~$24$ cores.

Then by the Chinese Remainder Theorem we combined the resulting
solutions with the logarithms modulo the~$38$ prime factors of~$N$ up
to bit-length~$43$, listed in Appendix~\ref{app:A}.  For the~$18$
larger ones of these moduli -- the smallest being~$\num{2252951}$ --
we computed the logarithms using the Lanczos algorithm as well.  For
the remaining~$20$ small factors, i.e., up to $\num{165313}$, we
simply generated the corresponding subgroups in the target field and
read off the discrete logarithms from the table.  This way, and making
use of the Galois orbits, we obtained the logarithms $\log(x + u)$ for
all factor base elements, where $u \in \F_{2^{30}}$.  The remaining
computations accounted for a total of~$\num{2105}$ core hours on a
single $16$-core computing node.


\section{Degree 2 Elimination Methods}\label{sec:degree2}

As it is the bottleneck, in order to have a feasible computation it is indispensable to have an extremely efficient degree 2 elimination
algorithm and implementation.
In this section we apply the sketch of degree 2 elimination given in~\S\ref{sec:GKZ} to our target field
and describe algorithms which are more efficient than those presented in~\cite{GGMZ13b,GKZ14a}.
Most of the techniques are more general and are presented in the setting of
an arbitrary finite field $\F_{q^k}$ and $x^{q+1}-x-\gamma=0$.
Let $Q = X^2 + q_1X + q_0 \in \F_{q^k}[X]$ be an arbitrary irreducible quadratic polynomial to
be eliminated, i.e., to be written as a product of linear elements.

\subsection{Direct Method}

Recall that $x^q = \frac {x + \gamma} x$, so that
\begin{align}\label{eqn:directmethod}
  x^{q+1} + a x^q + b x + c &= \tfrac 1 x \big( (b + 1) x^2 + (a + c + \gamma) x + \gamma a \big) \\
  \nonumber &= \frac {b + 1} x \bigg( x^2 + \frac {a + c + \gamma} {b + 1} x + \frac {\gamma a} {b + 1} \bigg) .
\end{align}
Since the denominator~$x$ on the r.h.s.\ of~(\ref{eqn:directmethod}) is in the factor base, we know its logarithm. There is also no lattice 
to consider as the quadratic term must be~$Q$, leading to the equations $q_1 = \frac {a + c + \gamma} {b+1}$ and $q_0 = \frac {\gamma a} {b+1}$. 
In order for the l.h.s.\ of~(\ref{eqn:directmethod}) to split over~$\F_{q^3}$, since there is only one Bluher value we must have 
$(b - a^q)^{q+1} = (c - a b)^q$. 
Writing $b = \frac 1 {q_0} (\gamma a - q_0)$ and $c = \frac 1 {q_0} (-q_0 a - \gamma q_0 + \gamma q_1 a)$ leads to a 
univariate polynomial equation in~$a$, namely
\[ (-q_0 a^q + \gamma a - q_0)^{q+1} - q_0 (-\gamma a^2 + q_1 \gamma a - \gamma q_0)^q = 0. \]
This equation may be solved by computing the GCD with $a^{q^3} - a$, by writing $a$ using a basis for $\F_{q^3}$ over $\F_q$ and solving 
the resulting quadratic system using a Gr\"obner basis computation~\cite{GGMZ13b}, or using the following much faster technique which exploits the 
fact that $B = 1$. 

Since $B = 1$ is the only Bluher value for $k = 3$, writing $b = \smash{\frac 1 {q_0}} (\gamma a - q_0) = u_0 a + v_0$ and 
$c = \smash{\frac 1 {q_0}} (-q_0 a - \gamma q_0 + \gamma q_1 a) = u_1 a + v_1$ as in~(\ref{eqn:abc}), the equation~(\ref{Bluhermaineqn}) becomes
\begin{equation}\label{eqn:elimination2direct}
(-a^q +u_0a+v_0)^{q+1}-(-u_0 a^2 + (u_1-v_0)a + v_1)^q=0.
\end{equation}
By setting $X_i=a^{q^i}$ we can rewrite this as $E=0$ with
\[ E \defeq (-X_1+u_0X_0+v_0)(-X_2+u_0^qX_1+v_0^q)
  - (-u_0^qX_1^2+(u_1^q-v_0^q)X_1+v_1^q) \]
where~$E$ is considered as a polynomial in the $X_i$.
Notice that the coefficient of $X_1^2$ vanishes, a consequence of $B=1$.
Raising~(\ref{eqn:elimination2direct}) to the power~$q$ while
substituting $X_i^q$ with $X_{i+1}$ (and $X_3$ with $X_0$)
yields the equation $\tilde E= 0 $ with
\[ \tilde E \defeq (-X_2+u_0^qX_1+v_0^q)(-X_0+u_0^{q^2}X_2+v_0^{q^2})
  - (-u_0^{q^2}X_2^2+(u_1^{q^2}-v_0^{q^2})X_2+v_1^{q^2}). \]
Since the factors $-X_2+u_0^qX_1+v_0^q$ coincide in $E$ and $\tilde E$
(essentially by construction) and $X_0$ occurs only in their cofactors, it
is easy to see that $E + u_0 \tilde E = \alpha_{12} X_1 X_2 + \alpha_1 X_1
+ \alpha_2 X_2 + \alpha$ with $\alpha_{*}$ being simple expressions in
$u_0, u_1, v_0, v_1$.  The special case $\alpha_{12} = 0$ can be handled
easily so we assume that $\alpha_{12} \ne 0$ and rewrite the equation
$\alpha_{12} X_1 X_2 + \alpha_1 X_1 + \alpha_2 X_2 + \alpha = 0$ as
\[ X_1 = -\frac {\alpha_2 X_2 + \alpha} {\alpha_{12} X_2 + \alpha_1} . \]
By raising this equation to the power~$q$ (with the usual substitution
of $X_i^q$ with $X_{i+1}$) one can express $X_2$ in terms of $X_0$,
and after doing this again $X_0$ in terms of $X_1$.  Substituting these
three equations into each other gives an equation of the form $X_0 =
\frac {\beta_2 X_0 + \beta} {\beta_{12} X_0 + \beta_1}$ or, equivalently,
\begin{equation}\label{eqn:degree2}
\beta_{12}X_0^2+(\beta_1-\beta_2)X_0-\beta=0.
\end{equation}
In the rare case that the l.h.s.\ of~(\ref{eqn:degree2}) is identically zero 
we abort the direct elimination method and apply the backup method; otherwise we check whether~(\ref{eqn:degree2}) has solutions in $\F_{2^{30}}$ and when it does so, we check whether they satisfy~(\ref{eqn:elimination2direct}). It is not easy to classify which solutions of~(\ref{eqn:degree2}) 
lead to solutions of~(\ref{eqn:elimination2direct}), although in practice almost all of them do so.

Experimentally we found that a randomly chosen~$Q$ is eliminable by this method with probability very 
close to $1/2$. A heuristic argument supporting this observation is that the derived equation~(\ref{eqn:degree2}) presumably behaves as a random degree~$2$ equation, which 
therefore has a solution in $\F_{2^{30}}$ with probability $\approx 1/2$. 

Before explaining the backup method we first introduce three efficient techniques that we employed for solving related problems.

\subsection{Finding a Degree 2 Elimination with Probability $q^{-2}$}\label{sec:qpowminus2}

A simple method for eliminating a degree~$2$ element $Q$ in $\F_{q^k}[X]$ consists of picking a random $a \in \F_{q^k}$, 
computing $b = u_0 a + v_0 = \smash{\frac 1 {q_0}} (\gamma a - q_0)$ and $c = u_1 a + v_1 = \smash{\frac 1 {q_0}} (-q_0 a - \gamma q_0 + \gamma q_1 a)$ as before,
and then checking whether the triple $(a,b,c)$ leads by~(\ref{eqn:Bluhertransform}) to a Bluher value $B \in {\mathcal B}_{k}$ using 
Theorem~\ref{thm:BluherRecurrence}, i.e., whether $P_k(\frac 1 B) = 0$ for the polynomials~$P_i$ defined by~(\ref{eqn:BluherRecurrence}). 
The success probability is expected to be about~$q^{-3}$ per trial by Theorem~\ref{thm:BluherCount}.

A better success rate is achieved by guessing a root~$r$ of the polynomial $X^{q+1} + a X^q + b X + c$ first (with~$b$ and~$c$ as
above) and testing if the polynomial splits completely. Note that when~$r$ is a root and we plug
in~$b$ and~$c$, we have $r^{q+1} + a r^q + (u_0 a + v_0) r +(u_1 a + v_1) = 0$ from which we recover
\[ a = - \frac {r^{q+1} + v_0 r + v_1} {r^q + u_0 r + u_1} \,. \]
From this we again compute $B = \frac {(b - a^q)^{q+1}} {(c - a b)^q} \in \F_{q^{k}}$ and check by
Theorem~\ref{thm:BluherRecurrence} whether $B \in {\mathcal B}_{k}$.
With this method the probability that the corresponding $B$ is a Bluher value has
been increased (heuristically) to about $q^{-2}$ by~\cite[Thm.~5.6]{bluher}.

The recurrence~(\ref{eqn:BluherRecurrence}) for computing
$P_k(\frac 1 B)$ as per Theorem~\ref{thm:BluherRecurrence} can be modified by
\begin{alignat*}{1}
&P^*_1 = 1 ,\ P^*_2 = b - a^q , \\
&P^*_i = (b - a^q)^{q^{i-2}} P^*_{i-1} - (c - a b)^{q^{i-2}} P^*_{i-2} ,
\end{alignat*}
where~$P^*_k$ provides (after clearing denominators) a polynomial~$P$
in~$r$ having the same roots as~$P_k(\frac 1 B)$, and which may be
used for the interpolation approach described next.

\subsection{Using Interpolation to Find Roots of Featured Polynomials}\label{sec:interpolation}

When encountering the task of finding a root in $\smash{\F_{q^k}}$ of a
polynomial~$P$ of high degree in one variable~$X$ which can be written
as a low degree polynomial in the variables $X, X^q, X^{q^2}, \dots$,
the method of picking random $r \in \F_{q^k}$ until $P(r) = 0$ can
sometimes be sped up as follows.

Let $r_0, r_1, \dots, r_{\ell} \in \smash{\F_{q^k}}$ such that
$r_1, \dots, r_{\ell}$ are linearly independent over $\F_q$ (in
particular $\ell \le k$) and let $\mathcal R \defeq \{ r_0 +
\sum_{i=1}^{\ell} c_i r_i \mid c_1, \dots, c_{\ell} \in \F_q \}$.
Then there is a low degree polynomial $\tilde P$ satisfying
$\tilde P(c_1, \dots, c_{\ell}) = P(r_0 + \sum_{i=1}^{\ell} c_i r_i)$;
let $D = \deg(\tilde P)$ and assume $D < q$.
The polynomial~$\tilde P$ can be computed by evaluating~$P$ at
$D + \ell \choose \ell$ elements of $\mathcal{R}$ and interpolating,
after which the~$q^{\ell}$ values $P(r)$ for $r \in \mathcal{R}$ can be
obtained by evaluating $\tilde P$.
Notice that still $q^{\ell}$ operations are needed to cover the set
$\mathcal{R}$ but if evaluating~$P$ is sufficiently more expensive
than evaluating $\tilde P$, this method is faster.
The optimal value of $\ell$ depends on $q$, $D$ and the evaluation costs for
the two polynomials.

\subsection{Using GCD Computations}

In certain situations we have to find an element in $\F_{q^k}$
which is a root of two polynomials~$P_1$ and~$P_2$.
In this case one can speed up the interpolation approach as follows.

Let $\tilde P_1$ and $\tilde P_2$ be the low degree polynomials
corresponding to $P_1$ and $P_2$ and let their degrees be
$D_1$ and $D_2$, respectively.
If a tuple $(c_1, \dots, c_{\ell})$ leads to a root of $P_1$ and $P_2$ then
$c_1$ is a root of $\tilde P_i(c_1, \dots, c_{\ell})$, $i=1,2$, considered
as a polynomial in the variable $c_1$.
Therefore $c_1$ is also a root of the greatest common divisor of these
two univariate polynomials.
If $q^k$ is much bigger than $D_1 D_2$, the degree of the GCD is usually
not bigger than~$1$ so that one trades~$q$ evaluations
of $\tilde P_1$ (and/or $\tilde P_2$) for one GCD computation.

\subsection{Backup Method}

When the direct method fails, which occurs with probability $\approx 1/2$, we use the following backup approach, based on the idea from~\cite{GGMZ13b}; for the remainder of this section we assume $\F_{q^k}=\F_{2^{30}}$.
Instead of using $q = 2^{10}$, $k= 3$ we use $q = 2^6$, $k = 5$, which by Theorem~\ref{thm:BluherCount} means there are $q^2 + 1 = 4097$ Bluher values,
but now the r.h.s.\ of the relevant equation will have higher degree and thus a smaller chance that the cofactor is $1$-smooth.

Let $y = x^{1024}$ and $\overline x = x^{16}$, so that $y = \overline x^{64}$. As $x^{1025} + x + \gamma = 0$ we have $x = \frac {\gamma} {y+1}$
and $\overline x = \smash{\big( \frac{\gamma}{y+1} \big)^{16}}$. We therefore have
\begin{align}\label{eqn:backup1}
  \nonumber \overline x^{65} + a \overline x^{64} + b \overline x + c
  &= y \big( \tfrac {\gamma} {y+1} \big)^{16} + a y + b \big( \tfrac {\gamma} {y+1} \big)^{16} + c \\
  &= \tfrac 1 {(y+1)^{16}} \big( a y^{17} + c y^{16} + (\gamma^{16} + a) y + (b \gamma^{16} + c) \big) .
\end{align}
Now let $\widetilde{Q}(X) \defeq (X+1)^2 Q\big(\frac{\gamma}{X+1}\big)$ so that $Q(x) = \frac{\widetilde{Q}(y)}{(y+1)^2}$, and consider the lattice
\[
L_{\widetilde{Q}} \defeq \big\{ (w_0, w_1) \in \F_{q^5}[X]^2 \mid w_0 + \tfrac {(X+1)^{16}} {\gamma^{16}} w_1 \equiv 0 \!\!\pmod{\widetilde{Q}} \big\}.
\]
In general, $L_{\widetilde{Q}}$ has a basis of the form $(X+u_0,u_1),(v_0,X+v_1)$ with $u_i,v_i \in \F_{q^5}$. Thus for $a \in \F_{q^5}$ we have
$(X + u_0 + av_0,aX +u_1 + av_1) \in L_{\widetilde{Q}}$. Substituting this element into $w_0,w_1$ and evaluating the resulting expression at $y$ 
leads to the r.h.s.\ of~(\ref{eqn:backup1}) being
\[
  \tfrac 1 {(y+1)^{16}} \big( a y^{17} + (u_1 + a v_1) y^{16} + (\gamma^{16} + a) y + \gamma^{16} (u_0 + a v_0) + u_1 + a v_1 \big)
\]
and thus $b = a v_0 + u_0$ and $c = a v_1 + u_1$. The l.h.s.\ of~(\ref{eqn:backup1}) transforms into a Bluher polynomial provided that 
$(a^{64} + b)^{65} = B (a b + c)^{64}$ for some $B \in \mathcal{B}_5$. This results in the equation
\begin{equation}\label{eqn:backupfora}
(a^{64} + v_0 a + u_0)^{65} + B(v_0 a^2 + (u_0 + v_1) a + u_1)^{64} = 0 \,.
\end{equation}
As before, $\F_{q^5}$-roots of~(\ref{eqn:backupfora}) can be computed, if they exist, via a GCD computation or using a basis for $\F_{q^5}$
over $\F_q$ and solving the resulting quadratic system using a Gr\"obner basis approach. 
Instead of solving~(\ref{eqn:backupfora}) directly, we used
the probability $q^{-2}$ method and interpolation with $\ell = 3$ so that
we expect to find about $q = 2^6$ completely splitting l.h.s.\ of~(\ref{eqn:backup1}) per interpolation.

For each of these we check whether the r.h.s.\ of~(\ref{eqn:backup1})
is $2$-smooth as follows.  Denote by~$R$ the polynomial corresponding
to its numerator and let~$m$ be a linear fractional transformation
over~$\F_{2^{60}}$ mapping the two roots of $\widetilde Q$ to~$0$
and~$\infty$.  Then the polynomial~$R$ becomes (up to a scalar)
$X^{16} - \alpha X$ for some $\alpha \in \F_{2^{60}}^{\times}$, and the
transformation~$\overline m$ maps~$R$ to $X^{16} - \overline \alpha X$,
where~$\overline \cdot$ denotes the $\F_{2^{60}} / \F_{2^{30}}$ Galois
conjugate, i.e., powering by $2^{30}$.  If and only if
$\alpha \in \F_{2^{60}}$ is a fifteenth power, does the r.h.s.\
of~(\ref{eqn:backup1}) split completely over $\F_{2^{60}}$ and is thus
$2$-smooth over $\F_{2^{30}}$.  As~$\widetilde Q$ is irreducible,
the transformation $\overline m^{-1} m$ exchanges~$0$ and~$\infty$
and is thus of the form $X \mapsto \frac \beta X$ for some $\beta \in
\F_{2^{60}}^{\times}$, while mapping $X^{16} - \overline \alpha X$ (up to a scalar)
to $X^{16} - \alpha X$ and hence $\alpha \overline \alpha = \beta^{15}$,
which implies that $\alpha = \smash{\alpha^{2+2^{30}}} \beta^{-15}$ is already
a third power.  Therefore, we expect the r.h.s.\ of~(\ref{eqn:backup1})
to be $2$-smooth with probability $\frac 1 5$, which is much higher
than for a random polynomial of this degree, thanks to it being
transformable to a Bluher polynomial.

Furthermore, whenever the r.h.s.\ is $2$-smooth it must factorise into a product of five linear polynomials and five degree two polynomials.
In particular, assume that the polynomial~$R$ on the r.h.s.\ of~(\ref{eqn:backup1}) splits completely over
$\F_{2^{60}}$.  Similarly as above, denote by~$m$ a linear fractional transformation mapping~$R$ to $X^{16} - X$ and the two roots of $\widetilde{Q}$ to 
$0$ and $\infty$.  Then the fifteen other roots of $R$ are $r_i=m(\zeta_{15}^i)$, $i = 0, \dots, 14$, with $\zeta_{15} \in \mu_{15}$ 
a fixed primitive fifteenth root of unity.  Notice that $\overline{\zeta_{15}}=\zeta_{15}^4$. 
In order to find the roots $r_i$ contained in the subfield $\F_{2^{30}}$
one has to solve $m(\zeta_{15}^i)=\overline{m(\zeta_{15}^i)}=\overline{m}(\zeta_{15}^{4i})$ or equivalently 
$\overline{m}^{-1} m(\zeta_{15}^i) = \zeta_{15}^{4i}$; 
let $n = \overline{m}^{-1} m$. Since~$n$ exchanges~$0$ and~$\infty$, and because $n \overline{n} = I_2$ holds, $n$~is actually a 
transformation of the form $X \mapsto \frac{b} X$ for some $b \in \F_{2^{30}}^{\times}$. Furthermore $n$~maps $X^{16} - X$ to 
itself which implies $n(\mu_{15}) \subset \mu_{15}$, hence
$b \in \mu_{15} \cap \F_{2^{30}}=\mu_3$. Therefore the equation $n(\zeta_{15}^i)=\zeta_{15}^{4i}$ 
from above becomes $\zeta_{15}^{5i} = b$ which has exactly five solutions. 
Thus the cofactor of $\widetilde{Q}$ splits into five polynomials of degree one and 
five of degree two. We remark that the same reasoning can be applied to determine the possible splitting pattern for a general Bluher 
polynomial in positive characteristic.

For each set of five degree two polynomials, we attempt to eliminate each of them using the direct method, which succeeds with probability $1/2^5$, assuming these eliminations are independent. One thus expects to eliminate such degree two elements after trying $160$ Bluher values, a figure which 
was borne out by our experiments. One also expects the approach to fail for a given irreducible of degree two with probability 
$(1 \!-\! \frac 1 {160})^{4097} \approx 7 \!\cdot\! 10^{-12}$.
If this happened, we simply restarted the elimination at an ancestor with a
different seed for randomness.
In terms of timings per degree two elimination, on average these ranged between $1.1$ms on some machines 
and just less than $1$ms on others.


\section{Small Even Degree Elimination Techniques}\label{sec:smalleven}

An overview of the various descent methods that we applied for the
smaller degrees is depicted in Figure~\ref{fig:small-degree}.  In this
short section we explain those methods for even degrees that are based on
degree~$2$ elimination, and in the next section we detail the methods
for odd degrees.

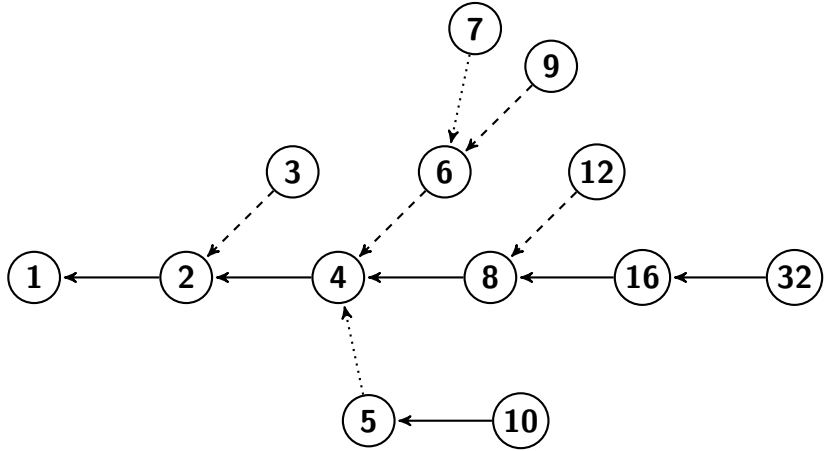
\begin{figure}[t]
  \begin{center}
    \begin{tikzpicture}[-stealth', shorten >=1pt, thick,
      every node/.style={circle, draw, font=\sffamily\large\bfseries}]
      
      \node (1) at (0, 0) { 1 };
      \node (2) at (2, 0) { 2 };
      \node (4) at (4, 0) { 4 };
      \node (8) at (6, 0) { 8 };
      \node (16) at (8, 0) { \!16\! };
      \node (32) at (10, 0) { \!32\! };
      
      \node (3) at (3.4, 1.4) { 3 };
      \node (6) at (5.4, 1.4) { 6 };
      \node (9) at (6.8, 2.8) { 9 };
      \node (12) at (7.4, 1.4) { \!12\! };
      
      \node (5) at (4.4, -1.9) { 5 };
      \node (7) at (5.8, 3.3) { 7 };
      \node (10) at (6.4, -1.9) { \!10\! };
      
      \draw (2) -- (1);
      \draw (4) -- (2);
      \draw (8) -- (4);
      \draw (16) -- (8);
      \draw (32) -- (16);
      \draw (10) -- (5);
      
      \draw[dashed] (3) -- (2);
      \draw[dashed] (6) -- (4);
      \draw[dashed] (9) -- (6);
      \draw[dashed] (12) -- (8);
      
      \draw[dotted] (5) -- (4);
      \draw[dotted] (7) -- (6);
    \end{tikzpicture}
  \end{center}
  \caption{Overview of the small-degree, `non-classical' descent methods.
    The encircled numbers represent the degrees of elements.  Solid
    arrows indicate two-to-one descents (including the backup strategy
    for $2 \to 1$), dashed arrows depict three-to-two descents, and
    dotted arrows represent specific odd-degree methods.}
  \label{fig:small-degree}
\end{figure}

\subsection{Degree 4 Elimination}

By Proposition \ref{prop:buildingblock} this case is reduced to
the elimination of a degree $2$ polynomial over $\F_{q^6}=\F_{2^{60}}$.
We use the probability $q^{-2}$ method and interpolation with $\ell = 1$
to obtain a polynomial $\tilde P$ of degree $15$ over $\F_{q^6}$ and
search for its zeroes in $\F_q$ as follows.
By choosing an $\F_q$ basis of $\F_{q^6}$ the polynomial $\tilde P$
can be expressed as a linear combination of six polynomials
$\tilde{P_1}, \dots, \tilde{P_6} \in \F_q[X]$ so that an $\F_q$-zero
of $\tilde P$ is also a zero of $\text{GCD}(\tilde{P_1}, \dots, \tilde{P_6})$.
In the rare case that the degree of this GCD is bigger than $1$ we
evaluate $\tilde P$ at all elements of $\F_q$.
Usually it is sufficient to compute the GCD of at most three of the
$\tilde{P_i}$.

\subsection{Degree~$2d$ to Degree~$d$}

For $2d \in \{8,10,16,32\}$ we rewrite an irreducible in $\F_{q^3}[X]$ of degree $2d$ into factors of degree (at most)~$d$, using the
algorithm outlined in the proof sketch of Proposition~\ref{prop:buildingblock}. This method is based on degree~$2$ elimination over the 
fields~$\F_{q^{3d}}$, which we perform using the technique already described in~\S\ref{sec:qpowminus2}, since the subsequent costs dominate 
the elimination costs by a large factor.


\section{Small Odd Degree Elimination Techniques}\label{sec:smallodd}

In this section we explain the methods employed for eliminating irreducible polynomials in $\F_{q^k}[X]$ of small odd degree~$d$ into polynomials of
degrees $d \!-\! 1, \dots, 1$.
In the same way that the degree $2$ elimination extends to any even degree
polynomial these techniques allow to eliminate irreducible polynomials of
degree $nd$ into polynomials of degrees $n (d \!-\! 1), \dots, n$.

\subsection{Degree 3 Elimination}

We show in this section how the degree~$2$ elimination technique can be extended to eliminate degree~$3$ polynomials. Instead of eliminating degree~$2$ polynomials into linear polynomials, we eliminate degree $3$ polynomials into degree~$2$ polynomials: it is a $3$--to--$2$ elimination.

As in~\S\ref{sec:GKZ} we use the field representation $\F_{q^{kn}} = \F_{q^k}[X] / (I)$, where~$I$ divides $h_1 X^q - h_0$.
We may assume that $h_0$ and $h_1$ have degree at most $1$, since in our setup $h_0 = X + \gamma$ and $h_1 = X$.
The $3$--to--$2$ elimination takes as input a cubic polynomial~$Q$ in $\F_{q^k}[X]$ and rewrites it into a product of linear and quadratic polynomials $Q_i \in \F_{q^k}[X]$ such that $Q \equiv \prod_i Q_i \bmod I$.

In our setup, the base case consists in eliminating cubic polynomials over $\F_{2^{30}}$ into quadratic and linear polynomials over $\F_{2^{30}}$, so $q = 2^{10}$ and $k = 3$.
As mentioned, this technique allows to eliminate irreducible polynomials of degree $3d$ over $\F_{2^{30}}$ into polynomials of degree $2d$ and $d$ over $\F_{2^{30}}$ in the same way that the degree $2$ elimination extends to any even degree polynomials.
Concretely, let $Q \in \F_{2^{30}}[X]$ be an irreducible polynomial of degree $3d$.  Then $Q$ splits in $\F_{2^{30d}}[X]$ into~$d$ Galois-conjugate irreducible factors, each of degree $3$.
Let $\hat Q$ be any of these factors. Let $N \colon \F_{2^{30d}}[X] \rightarrow \F_{2^{30}}[X]$ be the norm map. Then $Q = N(\hat Q)$. Applying the degree 3 elimination to $\hat Q$
yields a product $\prod_i \hat Q_i$ such that each $\hat Q_i$ is linear or quadratic in $\F_{2^{30d}}[X]$, and $\hat Q \equiv \prod_i \hat Q_i \mod I$.  Thus each $Q_i = N(\hat Q_i)$ is of degree
$d$ or $2d$ in $\F_{2^{30}}[X]$, and $Q \equiv \prod_i Q_i \bmod I$.

This method is used in our computation to eliminate polynomials of degree 6, 9 and 12.

\subsubsection{Degree 3 to Degree 2 Elimination}

Let $k \geq 3$.
We now show how to eliminate a cubic polynomial $Q$ over $\F_{q^k}$ into a product of linear and quadratic polynomials over $\F_{q^k}$.
As previously, let $\mathcal{B}_k$ be the set of all values $B \in \F_{q^k}$ such that $X^{q+1} - BX + B$ splits completely in $\F_{q^k}[X]$. Any polynomial of the form $X^{q+1}+aX^q + bX+c$ with $c \neq ab$ and $b\neq a^q$ splits completely in $\F_{q^k}[X]$ whenever $B = \frac{(b-a^q)^{q+1}}{(c-ab)^q}$ is in $\mathcal{B}_k$.
Consider the polynomial
\[ H_0(X, Y) = X Y + a Y + b X + c \in \F_{q^k}[X,Y] \,, \]
and define the transformation
\[ T_{\delta} \colon X \longmapsto \frac {X^2 + X + \delta} {X + 1} \,. \]
Let $H_{\delta}(X,Y) = H_0(T_{\delta}(X), T_{\delta^q}(Y))$.
Whenever $B = \frac{(b-a^q)^{q+1}}{(c-ab)^q} \in \mathcal{B}_k$, 
the numerator of $H_{\delta}(X,X^q)$ splits into linear and quadratic polynomials in $R_1 = \F_{q^k}[X]$. The denominator of $H_\delta$ equals $(X+1)(Y+1)$, and the polynomial $(X+1)(Y+1)H_\delta$ is of degree 2 in $Y$. Therefore the image of $(X+1)(Y+1)H_\delta$ in $R_2 = \F_{q^k}[X][1/h_1]$ has denominator $h_1^2$, so to get rid of all the denominators, we will work with the polynomial
\[ G_\delta = h_1^2(X + 1) (Y + 1) H_\delta \,. \]
On one hand, we want $G_\delta$ to split into linear and quadratic polynomials in $R_1$, i.e., $B$ should be an element of $\mathcal{B}_k$. On the other hand, we want $Q$ to divide the image of $G_\delta$ in $R_2$. The latter is of degree at most $2 + 2 \max(\deg h_0, \deg h_1)$. When $h_0$ and $h_1$ are of degree 1, and $Q$ divides the image of $G_\delta$ in $R_2$, the cofactor is of degree at most 1.

We shall now describe how to find suitable values of $\delta, a, b$ and $c$ in $\F_{q^k}$. First, observe that~$G_\delta$ is a linear combination of the polynomials
\begin{alignat*}{1}
I_0 &= h_1^2(X+1)(Y+1) T_\delta(X) T_{\delta^q}(Y),\\
J_0 &= h_1^2(X+1)(Y+1) T_{\delta^q}(Y),\\
K_0 &= h_1^2(X+1)(Y+1) T_\delta(X),\\
L_0 &= h_1^2(X+1)(Y+1),
\end{alignat*}
as $G_\delta = I_0 + aJ_0 + bK_0 + cL_0$. Let $I,J,K$ and $L$ be the reductions modulo $Q$ of the images in $R_2$ of $I_0, J_0, K_0$ and $L_0$ respectively; in particular, these are polynomials in $R_2$ of degrees at most~2. Then, $Q$ divides $G_\delta$ in $R_2$ if and only if $I + aJ + bK + cL = 0$. The latter is a linear system for the three variables $a,b$ and $c$, in the 3-dimensional vector space of polynomials of degree 2. Solving this linear system with $\delta$ as a symbolic variable yields a solution $a,b,c \in \F_{q^k}[\delta]$.
Choosing random values for $\delta$, the resulting $B$ is expected to fall in $\mathcal{B}_k$ with an expected probability of~$q^{-3}$. For $k = 3$, there is only one element in $\mathcal{B}_k$, so it might happen that no good value of $\delta$ exists. In that case, we start again with the transformation
\[ T_{\delta, \alpha} \colon X \longmapsto \frac {X^2 + X + \delta} {X + \alpha} \,, \]
in place of $T_{\delta}$, for random values $\alpha \in \F_{q^k}$.

\subsubsection{Optimising the Elimination}

Suppose we have computed $a, b, c \in \F_{q^k}[\delta]$, and are looking for values of~$\delta$ that give rise to a Bluher value.
Instead of trying random values for $\delta$, one can directly find these which yield a~$B$ in $\mathcal B_k$ via the characteristic polynomial of inverse Bluher values as per Theorem~\ref{thm:BluherRecurrence}.
By using the recurrence in~\S\ref{sec:qpowminus2} and interpolation (cf.~\S\ref{sec:interpolation}, in our practical setting we use $\ell = 1$) the random guessing approach could be sped up considerably.

\subsection{Eliminating Polynomials of Degree 5 and Degree 7}

Let~$Q$ be a polynomial of degree~$d$.
Consider two polynomials $F = \sum_{i=0}^{d_F} f_i X^i$ and
$G = \sum_{i=0}^{d_G} g_i X^i$ in $\F_{q^k}[X]$ with $d_F, d_G < d$.
We have
\[ F^qG-FG^q \,\equiv\, \frac{1} {h_1^{d-1}} \Big( F^{(q)}
  \big( \tfrac{h_0} {h_1} \big) h_1^{d-1} G - F G^{(q)}
  \big( \tfrac{h_0} {h_1} \big) h_1^{d-1} \Big) \pmod{I}, \]
with the numerator of the r.h.s.\ of degree at most $2 d \!-\! 2$.
The coefficients of this numerator are simple polynomials in the $f_i$,
$f_i^q$, $g_i$ and $g_i^q$, in particular, they are linear in each of these
variables.
If the polynomial~$G$ is fixed, it is easy to find all polynomials~$F$ such
that the r.h.s.\ is zero modulo~$Q$.
Indeed, for an $\F_q$-basis $(\alpha_i)$ of $\F_{q^k}$ one can write
$f_i = \sum_{j=1}^k f_{ij} \alpha_j$ with $f_{ij} \in \F_q$ (implying
$f_i^q = \sum_{j=1}^k f_{ij} \alpha_j^q$) and reducing the r.h.s.\ modulo $Q$
to obtain a linear system of $d k$ equations in the $d k$ unknowns $f_{ij}$.
Since $F=G$ is always a solution, the corresponding matrix has rank at
most $d k \!-\! 1$; if its rank is smaller, a non-trivial pair $(F, G)$ exists
which leads to an elimination of $Q$ into polynomials of degrees smaller
than~$d$.

Heuristic arguments suggest that the probability of finding a non-trivial
pair $(F, G)$ for a randomly chosen~$G$ is about $q^{-2}$.
Namely, discarding the $1 + (q \!+\! 1) (q^{d k} \!-\! 1)$ trivial solutions
$\{ (F, G) \mid c_F F = c_G G \ \text{for some} \ c_F, c_G \in \F_q \}$, one
expects that about a $q^{-d k}$ part of the roughly~$q^{2 d k}$ remaining
pairs $(F, G)$ give rise to a r.h.s.\ divisible by~$Q$.
Since one non-trivial solution $(F, G)$ entails $q^2 \!-\! q$ non-trivial
solutions $(a F + b G, G)$, with $a, b \in \F_q$, $a \ne 0$, the
probability of finding a non-trivial solution for a fixed~$G$ is
heuristically about $q^{-2}$.

The cost of the algorithm as presented is $O(q^2 (d k)^3)$ (with a
lower exponent if fast matrix multiplication is used; since the
subsequent costs of the elimination dominate, we did not do this).
If~$q$ is big compared to $d k$, the following variants might be
advantageous.  If the rank of the matrix corresponding to the linear
system obtained from a randomly chosen~$G$ has rank at most
$d k \!-\! 2$ then all its minors vanish.
Thus, instead of directly computing the rank of the matrix, one can
check whether a few minors vanish and, if so, compute the rank.
Since the minors are polynomials of degree $d k \!-\! 1$ in the
$g_{ij} \in \F_q$ defined by $g_i = \sum_{j=1}^k g_{ij} \alpha_j$, the
interpolation approach as well as the GCD approach may be used.
For an $\ell = 1$ interpolation the cost is $O(q (dk)^4)$ and for
$\ell = 2$ it is $O((dk)^5 \!+\! q(dk)^2)$.

In our computation we used this method for degrees~$5$ and~$7$ for
which the $\ell = 2$ interpolation was optimal ($q = 1024$ and
$d k = 15, 21$).  We remark that the subsequent costs for this method
are much bigger than for the Gr\"obner basis method, but this was not
a significant concern as the resulting costs for these degrees were
relatively inexpensive.


\section{Computing Roots of Bluher Polynomials and Their Transformations}\label{sec:BluherRoots}

In order to find all roots of a Bluher polynomial or a transformation
thereof, which is known to split completely, we use two different methods.

If a transformation is known which maps the polynomial to an explicit
polynomial with known roots, this transformation is used to obtain the
roots. This applies to the following cases:
\begin{itemize}
\item the direct method in degree~$2$ and degree~$3$ elimination,
  where we transform to the only Bluher polynomial $X^{q+1}-X+1$ and
  use its pre-computed roots;
\item the backup method in degree~$2$ elimination for which the
  r.h.s.\ polynomial can be transformed to $X^{16} - \beta^{15} X$;
\item and the degree $n d$ to $n (d \!-\! 1), \dots, n$ elimination
  where the splitting of the polynomial is obvious.
\end{itemize}

In the other cases of degree $2^t$ elimination, one root of the
polynomial is known so that by sending this root to~$\infty$ we have
to find the roots of a polynomial $\tilde P = X^q + a_1 X + a_0$.
The case $a_1 = 0$ is trivial so we assume that $\tilde P$ has no
multiple roots.  Let $r_0, r_1 \in \smash{\F_{q^k}}$ be random elements
and compute $s_1 X + s_0 \equiv \sum_{i=0}^{k-1} (r_1 X + r_0)^{q^i} \pmod
{\tilde P}$.  By construction $(s_1 X + s_0)^q \equiv s_1 X + s_0
\pmod{\tilde P}$ holds so that $\tilde P$ divides $(s_1 X + s_0)^q
- (s_1 X + s_0)$.  If $s_1$ does not vanish, the roots of $\tilde P$
are $\frac {\beta - s_0} {s_1}$, $\beta \in \F_q$; otherwise one tries
other random pairs $(r_0, r_1)$.

Finally, for the degree $6$, $9$, $10$ and $12$ eliminations, the
polynomials are also a transformation of a Bluher polynomial. Three
roots of the Bluher polynomial are found with the Cantor--Zassenhaus
algorithm, from which all the roots of the original polynomial are
recovered.


\section{Classical Descent}\label{sec:classical}

As stated in the introduction, we set ourselves a discrete logarithm
challenge using the digit expansion (to the base~$2^{30}$) of the
mathematical constant~$\pi$, which is viewed, as is common, as a
source of pseudorandomness.  Specifically, the target element
$h_{\pi} \in \F_{2^{30750}}$ is represented by a polynomial
over~$\F_{2^{30}}$ of degree~$\num{1024}$.

In the individual logarithm phase, or the `descent' phase, of the
index calculus method, we seek to rewrite this polynomial, modulo
$X^{1025} + X + \gamma$, step-by-step as a product of polynomials
having smaller degrees, until we have expressed the target element as
a product of factor base elements.

\subsection{Descent Analysis}

Performing a descent starting from degree~$\num{1024}$ constitutes a
substantial challenge.  In order to optimise the overall computational
cost we employ a bottom-up approach in our analysis.  This means that
for every degree $d = 2, 3, 4, \dots$ we determine the expected cost
of computing a logarithm of an element represented by a polynomial of
degree~$d$.  This cost is the sum of finding a (good) representation
as a product of lower degree polynomials (the `direct cost'), plus the
expected total cost of the factors appearing in the product, based on
the earlier cost estimations for degrees~$<d$ (the `subsequent cost').

When computing the expected direct and subsequent costs, one needs
statistical information about the degree pattern distribution and the
resulting cost distribution for a random polynomial over~$\F_{2^{30}}$
of a given degree.  These can be obtained using formulas for the
number of irreducible polynomials and an approach involving formal
power series.  As the analysis becomes more difficult when the degrees
increase, we round the input costs appropriately and sometimes replace
the field~$\F_{2^{30}}$ by the field~$\F_2$ (without much loss of
accuracy).  Assisted by Magma for the power series computations, we
list in Table~\ref{tab:cost} the expected costs for each of the
degrees up to~$40$.%
\footnote{In common descent algorithms, a smaller degree usually means
  a lower cost, which however is not generally true in our case.  So
  in fact, our bottom-up analysis starts with the degrees for which the
  non-classical descent methods apply, cf.~Fig.~\ref{fig:small-degree}.}

\begin{table}
  \begin{center}
    \caption{Rounded cost estimations used in the descent phase.  The cost
      unit is normalised to the equivalent of one degree 16 elimination.}
    \label{tab:cost} \small
    \begin{tabular}{ccl}
      deg & cost & method and degrees \\\hline
      2 & 0 & two-to-one: $\leadsto$ 1 \\
      3 & 0 & three-to-two: $\leadsto$ 2, 1 \\
      4 & 0 & two-to-one: $\leadsto$ 2 \\
      5 & 0 & five-to-four: $\leadsto$ 4, 3, 2, 1 \\
      6 & 0 & three-to-two $\leadsto$ 4, 2 \\
      7 & 0 & seven-to-six: $\leadsto$ 6, 5, 4, 3, 2, 1 \\
      8 & 0 & two-to-one: $\leadsto$ 4 \\
      9 & 0.3 & three-to-two: $\leadsto$ 6, 3 \\
      10 & 0.3 & two-to-one: $\leadsto$ 5 \\
      11 & 10 & classical: $\leadsto$ 86, 70 \\
      12 & 0.5 & three-to-two: $\leadsto$ 8, 4 \\
      13 & 44 & classical: $\leadsto$ 100, 71 \\
      14 & 120 & classical: $\leadsto$ 99, 71, or 115, 72 \\
      15 & 160 & classical: $\leadsto$ 114, 72 \\
      16 & 1 & two-to-one: $\leadsto$ 8 \\
      17, 18 & 250 & classical: $\leadsto$ 73, 120 \\
      19, 20 & 360 & classical: $\leadsto$ 81, 119 \\
      21, 22 & 480 & classical: $\leadsto$ 89, 118 \\
      23, 24 & 630 & classical: $\leadsto$ 97, 117 \\
      25, 26 & 790 & classical: $\leadsto$ 105, 116 \\
      27, 28 & 960 & classical: $\leadsto$ 113, 115 \\
      29, 30 & \num{1150} & classical: $\leadsto$ 121, 114 \\
      31 & \num{1350} & classical: $\leadsto$ 129, 113 \\
      32 & \num{1000} & two-to-one: $\leadsto$ 16 \\
      33, 34 & \num{1580} & classical: $\leadsto$ 137, 112 \\
      35, 36 & \num{1830} & classical: $\leadsto$ 145, 111 \\
      37, 38 & \num{2000} & classical: $\leadsto$ 116, 147 \\
      39, 40 & ~\num{2300}~ & classical: $\leadsto$ 122, 148 \\
    \end{tabular}
  \end{center}
\end{table}

\subsection{Initial Split and Classical Descent}

Having finished the bottom-up analysis, starting from lowest degrees,
the actual descent computation was performed top-down, starting with
the target element of degree~$\num{1024}$.

In the first phase, referred to as `initial split', the target
element~$\beta$ is rewritten as a fraction
\[ g^i h_{\pi} = r(x) / s(x) , \] where
$\deg r + \deg s = \num{1024}$ and the integer~$i$ should be chosen so that
the polynomials $r, s \in \F_{2^{30}}[X]$ have a favourable factorisation
pattern.  The descent analysis suggested that a split with unbalanced
degrees is preferable, so we searched for splittings with
$\deg r = 205$ and with $\deg r = 256$.  The most promising initial
split we found was using $\deg r = 205$, $\deg s = 819$ and where
$i = \num{47611005802}$, resulting in a factorisation with largest
irreducible factor being of degree~$672$.

Next the `classical descent' was performed, which is based on the
following rewriting method.  Choose some $a \in \{ 0, 1, \dots, 10 \}$
and define $\overline x \defeq x^{2^{10 - a}}$ and $y \defeq \overline
x^{2^a}$, so that $\overline x = (\gamma / (y \!+\! 1))^{2^{10 - a}}$.
Then for polynomials $u, v \in \F_{2^{30}}[X]$ we have
\[ u(\overline x^{2^a}) \overline x + v(\overline x^{2^a})
  = u(y) \big( \tfrac {\gamma} {y + 1} \big)^{2^{10 - a}} + v(y) . \]
If a target field element represented by a polynomial
$Q \in \F_{2^{30}}[X]$ of degree~$d$ is to be eliminated, we choose
the polynomials $u, v$ such that~$Q$ divides one of the sides.  Now
balancing the degrees such that $\deg u + \deg v = d + e$ (for
$e \in \{ 0, 1 \}$) gives $2^{30 (e + 1)}$ trials for obtaining a
smooth cofactor and a smooth other side.

We refer to Table~\ref{tab:cd} for computational details of the
initial split and the classical descent.  This phase resulted in a
number of polynomials to be eliminated by the non-classical descent
methods, the timings of which are summarised in
Table~\ref{tab:clever}.

\begin{table}
  \begin{center}
    \caption{Initial split and classical descent: schedule and
      running times.} \label{tab:cd} \small
    \begin{tabular}{ccrrrr}
      is & when & degree & \,\#polys\, & \,iterations\, & core hours \\\hline
         & 09 Sep - 20 Sep 2016 & 1024 & 1 & 60 G & \num{248140}
    \end{tabular} \medskip\par
    \begin{tabular}{ccrrrr}
      cd & when & degree & \,\#polys\, & \,iterations\, & core hours \\\hline
  & 13 Oct - 30 Oct 2016 & 672 & 1 & 80 G & \num{220406} \\
       & 31 Oct - 02 Dec & 466 & 1 & 80 G & \num{395216} \\
       & 08 Dec - 03 Jan & 130 & 1 & 80 G & \num{121988} \\
       & 17 Dec - 26 Dec & 106 & 1 & 40 G & \num{120821} \\
  & 14 Jan - 22 Jan 2017 & 64 & 1 & 40 G & \num{85721} \\
       & 22 Jan - 28 Jan & 45 & 1 & 26 G & \num{52127} \\
       & 28 Jan - 01 Feb & 40 & 1 & 32 G & \num{56721} \\
       & 01 Feb - 02 Feb & 38 & 1 & 16 G & \num{26620} \\
       & 03 Feb - 04 Feb & 35 & 1 & 16 G & \num{23298} \\
       & 04 Feb - 07 Feb & 34 & 3 & 48 G & \num{69144} \\
       & 08 Feb - 11 Feb & 31 & 3 & 48 G & \num{70381} \\
       & 11 Feb - 14 Feb & 30 & 3 & 96 G & \num{143208} \\
       & 18 Feb - 21 Feb & 29 & 2 & 32 G & \num{47987} \\
       & 22 Feb - 04 Mar & 28 & 4 & 64 G & \num{96933} \\
       & 05 Mar - 08 Mar & 27 & 3 & 48 G & \num{72982} \\
       & 10 Mar - 13 Mar & 26 & 2 & 32 G & \num{38565} \\
       & 14 Mar - 18 Mar & 25 & 4 & 40 G & \num{43820} \\
       & 18 Mar - 26 Mar & 24 & 8 & 100 G & \num{110932} \\
       & 29 Mar - 01 Apr & 23 & 5 & 60 G & \num{59165} \\
       & 01 Apr - 09 Apr & 22 & 9 & 108 G & \num{108868} \\
       & 11 Apr - 24 Apr & 21 & 12 & 173 G & \num{153994} \\
       & 24 Apr - 08 May & 20 & 12 & 216 G & \num{192121} \\
       & 12 May - 30 May & 19 & 19 & 486 G & \num{347169} \\
       & 31 May - 09 Jul & 18 & 43 & \num{1143} G & \num{788765} \\
       & 19 Jul - 02 Aug & 17 & 39 & \num{879} G & \num{588148} \\
       & 05 Aug - 11 Oct & 15 & 116 & \num{2761} G & \num{1780136} \\
  & 14 Oct - 19 May 2018 & 14 & 209 & \num{6558} G & \num{4049638} \\
       & 20 May - 10 Sep & 13 & 592 & \num{6961} G & \num{4361593} \\
  & 10 Sep - 23 Jan 2019 & 11 & 1687 & \num{4647} G & \num{2851369} \\\hline
      & \emph{total running time} & & & & \num{17077836}
    \end{tabular}
  \end{center}
\end{table}
      
\begin{table}
  \begin{center}
    \caption{Timings of the non-classical descent methods.}
    \label{tab:clever} \small
    \begin{tabular}{rrr}
       degree & ~\#polys~ & core hours \\\hline
       32 & 8 & \num{2800053} \\
       16 & \num{8622} & \num{2944153} \\
       12 & \num{5644} & \num{976533} \\
       10 & \num{4694} & \num{790316} \\
       9 & \num{4184} & \num{372762} \\
       8 & \num{3998} & \num{1388} \\
       7 & \num{3671} & \num{236195} \\
       6 & \num{3467} & 615 \\
       5 & \num{3384} & 719 \\
       3 & \num{3744} & 8 \\
       0,1,2,4 & \num{22973} & 2 \\\hline
      \emph{total} & & \num{8122744}
    \end{tabular}
  \end{center}
\end{table}


\section{Reflections on the Security of the Proposal of Canteaut~\ea}\label{sec:canteaut}

One question which has a direct bearing on the security of the secure compressed encryption scheme of Canteaut \emph{et al.}~\cite{Canteaut}, is what is 
the complexity of solving discrete logarithms in the fields $\F_{2^{n}}$ with $n$ prime and $\approx \num{16000}$? 

In fact $n \approx \num{10322} > 2^{80/6}$ would suffice for their security condition, which is that the first stage of index calculus -- 
which has complexity $O((2^{\log_2 n})^6)$ when using the approach of Joux-Pierrot -- should have $80$-bit security. Firstly, one can reduce this 
slightly by using a degree $n$ irreducible factor of $h_1(X^q) X - h_0(X^q)$ with $\text{deg}(h_1) = 2$ and $\text{deg}(h_0) = 1$, allowing 
$q=2^{13}$ to be used rather than the stipulated $2^{14}$. Secondly, 
ignoring the descent is an overly conservative approach as one is then effectively assuming that logarithms can be computed in polynomial time, which 
partly explains why the performance of the scheme is so slow. Thirdly, one can instead use $q = 2^{12}$ and $h_1 \defeq X^3 + X^2 + X + 1$, 
$h_0 \defeq X^5 + X^4 + X^2$ as $h_1(X^q) X - h_0(X^q)$ has an irreducible factor of degree $\num{10333}$, 
which is the next prime after $\num{10322}$. For such 
$h_1,h_0$, the analysis of Joux-Pierrot does not apply in its entirety, but the alternative used by Kleinjung for the computation of discrete logarithms 
in $\F_{2^{1279}}$~\cite{1279Ann} gives an $O(q^7)$ algorithm for computing the logarithms of all irreducible elements of $\F_{q^{10333}}$ of 
degree $\le 4$.

Given that a target discrete logarithm is in $\F_{2^{10333}}$, any irreducible elements of degree~$d$ obtained during the descent can be split into some
degree $d / \text{GCD}(d, 12)$ irreducibles over $\F_q$ (or a subfield).  Hence for the initial split one can expect to obtain only elements of degree,
say $< 200$, in a reasonable time.  Furthermore, since the logarithms of elements of degree up to~$4$ are known, $d / \text{GCD}(d, 12)$ need only be a small power
of~$2$ times $1$, $2$, $3$ or $4$ in order to apply the GKZ step, so that, e.g., $d = 288$, $384$, $576$ or $768$ are also favourable degrees to search for during the initial split.
Note that since the degrees of $h_i$ are larger than in the present case, finding cofactors that are $1$-smooth will be more costly for each elimination. 
However, whether an individual logarithm can be computed in this field more efficiently than the factor base logarithms would require a detailed model and 
analysis of both parts, and we leave this as an open, but not particularly pressing question.

More generally, finding the optimal field representation, factor base logarithm method, and descent strategy for any given extension degree is an interesting open problem, with the understanding that optimal here means subject to the current state of the art.


\section{Concluding Remarks}\label{sec:conclusion}

We have presented the first large-scale experiment which makes essential recursive use of the elimination step of the GKZ quasi-polynomial algorithm and 
in doing so have set a new discrete logarithm record, in the field $\F_{2^{30750}}$. We have contributed novel algorithmic and arithmetic improvements 
to degree two elimination -- the core of the GKZ algorithm -- as well as to other small degree eliminations, which together made the computation feasible. 

Regarding open problems, the remaining major open problem in fixed characteristic discrete logarithm research is whether or not there is a polynomial time 
algorithm, either rigorous or heuristic. Since we now have $L(o(1))$ algorithms, one might hope that $L(0)$ complexity is not only possible, but may be discovered in the not too distant future. 


\section*{Acknowledgements}

The authors would like to thank EPFL's Scientific IT and Application Support (SCITAS) as well as the School of Computer and
Communication Sciences (IC) for providing computing resources and support.
The first and fourth listed authors were supported by the Swiss National Science Foundation via grant number 200021-156420.


\begin{subappendices}
\renewcommand{\thesection}{\Alph{section}}
  
\section{}\label{app:A}

This Magma script displays the partial factorisation of the group
order $N = 2^{30750} \!-\! 1$ by listing the known prime factors of
bit-length up to~$135$, in increasing order.

{\scriptsize
\begin{verbatim}
M := [
   < 1, 1.5850, 3^2 >,                                   < 21, 21.103, 2252951 >,      
   < 2, 2.8074, 7 >,                                     < 22, 21.488, 2940521 >,      
   < 3, 3.4594, 11 >,                                    < 23, 21.862, 3813001 >,      
   < 4, 4.9542, 31 >,                                    < 24, 21.890, 3887047 >,      
   < 5, 6.3750, 83 >,                                    < 25, 21.945, 4036451 >,      
   < 6, 7.2384, 151 >,                                   < 26, 23.333, 10567201 >,     
   < 7, 7.9715, 251 >,                                   < 27, 27.091, 142958801 >,    
   < 8, 8.3707, 331 >,                                   < 28, 27.294, 164511353 >,    
   < 9, 9.2312, 601 >,                                   < 29, 27.775, 229668251 >,    
  < 10, 9.5294, 739 >,                                   < 30, 31.188, 2446716001 >,   
  < 11, 9.5527, 751 >,                                   < 31, 31.342, 2721217151 >,   
  < 12, 10.266, 1231 >,                                  < 32, 33.040, 8831418697 >,   
  < 13, 10.815, 1801 >,                                  < 33, 36.030, 70171342151 >,  
  < 14, 11.136, 2251 >,                                  < 34, 36.276, 83209081801 >,  
  < 15, 11.984, 4051 >,                                  < 35, 37.969, 269089806001 >, 
  < 16, 12.587, 6151 >,                                  < 36, 39.774, 940217504251 >, 
  < 17, 13.706, 13367 >,                                 < 37, 40.044, 1133836730401 >,
  < 18, 15.586, 49201 >,                                 < 38, 42.576, 6554658923851 >,
  < 19, 16.621, 100801 >,
  < 20, 17.335, 165313 >,

  < 39, 50.143, 1243595348645401 >,
  < 40, 53.551, 13194317913029593 >,
  < 41, 55.544, 52546634194528801 >,
  < 42, 56.364, 92757531554705041 >,
  < 43, 57.302, 177722253954175633 >,
  < 44, 62.031, 4710883168879506001 >,
  < 45, 70.172, 1330118582061732221401 >,
  < 46, 70.838, 2110663691901109218751 >,
  < 47, 72.225, 5519485418336288303251 >,
  < 48, 74.080, 19963778429046466946251 >,
  < 49, 76.045, 77939577667619953038001 >,
  < 50, 81.503, 3427007094604641668368081 >,
  < 51, 81.757, 4086509101824283902341251 >,
  < 52, 93.333, 12477521332302115738661504201 >,
  < 53, 101.53, 3655725065508797181674078959681 >,
  < 54, 104.02, 20518738199679805487121435835001 >,
  < 55, 112.04, 5346075695594340248521086884817251 >,
  < 56, 114.78, 35758633131596900685051378954141001 >,
  < 57, 118.65, 519724488223771351357906674152638351 >,
  < 58, 134.30, 26815123266670488105926669652266381711401 > 
];
\end{verbatim}}


\section{}\label{app:B}

The following Magma script verifies the solution of the chosen DLP.

{\scriptsize
\begin{verbatim}
F2 := GF(2);
F2T<T> := PolynomialRing(F2);
F2_30<t> := ext< F2 | T^30 + T + 1 >;
F2_30X<X> := PolynomialRing(F2_30);
Xmod := X^1025 + X + t^3;
Fqx<x> := ext<F2_30 | Xmod >;

g := x + t^9;

pi := Pi(RealField(10000));
hpi := &+[ (Floor(pi * 2^(e+1)) mod 2) * t^(29-(e mod 30)) * x^(e div 30): e in [0..30749]];

log := 434390305789220128646032802928291982609103161559747835324845605488213355998660010229578378323680044\
0571584572618313992806371868198456791470780901419468919256264610453977016300076786167025822973048136833959\
0878467665466064152674655005652391637172013760770875176099662721206898645595263358736804087462977533382123\
2315131302891705001977748686889293319050071302409681861799652568296307597232351849866431934746650875959245\
0602671902187767401641781122604884838159314317487374378474672999763614677474730173878428493505419425791684\
6080161797788800050724547671213175949346909556874373705952104416698537444077761851591676796872323549138052\
3860665679424853707376856018873612183306770743192368627402259495347102152380263014405374835550410654239197\
2667845361799733363853947533351689470293301908054825769739572700652663971847149309353094674466286132575048\
6187657364509285871639479941370825143070270437660054019341835416260494645090240963521011103532544777046509\
1536191332904005087393896035530112997428238902910085725829585776641211605578600005263072054373393644877392\
6092004218334931938197619823713312568419369762209936082131923515369320441227699207625633325372043664307456\
4913221288187928934798264321289556576125416743751949748194132954432802514401212656718856068337071850239970\
2404331435381835036334725327407131866871748723856409267697576442746918982975062281036365566776392363298081\
2441781745674206200681802792267540028884393074308326228892657039366990569450742219627256186766938021938287\
9603656906564416733268275319957309740551923358198579176153287555764242904421472114283709899984255314466433\
0404904384027745746605272709406214394077281827572339426136219640847421087300602695383542220029909328375886\
1578377116422763381665249724400816279085274871402772270510531188789749167205672072687459257990843335596742\
6187686671538912268649519154377368715898151933453098793172504084315911799081035232913127062198736711334781\
0717626923247021660901044216022764408942571761295771074530551783174567208344889504350258107462608534235119\
4675706868776366076465146991376300418892571221113976403998862864418344276135357701865874422973325150079746\
7501768483718237370025106218993825551076290245046680872681272148350815794759887912260683290899795915060880\
6013703943664419505621194228118231762212387888732953199508746830128029948251456319475543776092009178577716\
1082923337630423567924243041896820260802567577362671623460935438739035558414434183224164944201556049287907\
7991851307804633615629952757339246701775647067870966974093224634621910071023645348714835033340680626583962\
7181858720737793092568995339096500863883774440128094638960588682858167109320517538894030840557729910897452\
9052457258146806504479844312652933963934345039782920981316618021343753492261904833585652640510839540488804\
2612252670230526071211952938402703448682872325061917710133401179762443544506255485207598674924919709543956\
0480318790321161740832292030651151185976169256536200664790380139143032263888546981729200401142632103111392\
4064098599256006448009481354912118594842305295497311552816436214605893674831740404346654366643331475581457\
0822194809284724892831668254137875554607269348302940906636745391929777748817519335123765508977700051204383\
7533687288787877793156906985983015300658329325070253772068595074557067328255774593601995026745153779206686\
9650673586669794934924840022405511994813571337865722286519034198754214510515702993605481375015502086300325\
6018159355825032772167691770136544296756481710603864268227290506733859544325189994664233703918695556790811\
1315610210026763626186928192211646594444420509444892574914158485329039723865366746184767175008863773048656\
3446785524791116048212398255867963984760609741691455610921644482449200626753586774752932974017167339582031\
0790239779275945029198240056627738432662859945002357653451740145284785560911363996991203626319546870871705\
6808021878700216950298153005486414427596641552441541768460669711145598361205051763664490759041169284429599\
0365865929732392357857132459614529052060995248089072034429436921566507819543803237192871948024141916518736\
8418902014641978184855013104500221707838776463634643068277947172681309403115147907824510756707893774449421\
9207948930511720262176341738723606892094105170119994279261501743156994561213126256648627608227947223213849\
5013861890264189122129401515436677143878765049187831537840274254424382439643469154213721491144269067296361\
6051432753779913538873113633936518574532486982523322988706677914219270407026556650524538565864371064774748\
4867602003344929825895687057572953680017482076645977509597091659601823250525835336256115596721455871414374\
6351560051626179272980584314925670111185042232817880467226026682913911023733784867852189640489302773537683\
5717196907350765540368510887811021602414497973745995443737338335533027023950567389671881860493290854194750\
4545177330607937383529015834991073224632423885343798187392846024451158898698310032950777455747923800905182\
5392779456539720775448781391228678148890742424968562845987995586524902058280079962717162582155370835454789\
1650133177176132154040253188770800869123676869923554635172224567587671587185280855086810932854774829377859\
5404816312158436596143115062781708882448796023559409240906080947320967881759405195787139063876234354405964\
8235535589706649686543556212160299922543798489873871039814453052217419176811991881558223042302400904599466\
1066469178877447920337377742826574991181933695267091497918364089854392183039262391590917835160240205653970\
7517008929290929045177138658046268536040641420459151077068958543533667447527529852899548554504385661263910\
5169467841260768876286409103241121712596014243949615914988012323302604935241664873668009690540503022580253\
0092886885087913039140486583595252772751948108253829710122682656841190662323010434904953320922715075445544\
6953400512189316233150957313253295522784402751547898491983980946849472304171863600101812018288349972305434\
3675142359984931348757366084586724591785291993492794862287163922674152530252243396245896594675979483293876\
9882519662053824434839355757065422108074744179224763763393756442105145442912635588510161124874921305214834\
2613297513712007218064275799534932941517688671218664648593181312993033897256815354597112618020047499280393\
9314831200808631865510984083167559533814842373674105290245276488614557715526996385085352272913099523432935\
5790186423594446518896790756017039684508062634380893783856349512454804012364368801029663757192460525897738\
2957221568396487363503721741774449928007113090443864568563345791384313987988281295492162647792435761530022\
8229477443941100330578127970839153609400825917728711687262962264320440663300044409692211867133192177593701\
8843683810819813013716303795350336907369442796944580198184702129321621587922971310761603777486510459327119\
2098239336790041115235586652951234405020171727222920596506810093182173017789351842956657267087708147399802\
4289560797711850480242923609318460380224180343465101563776259867176310949318750072922431757429976110774286\
8319099581165622517990548728856574852630398033109865849741898667555480248902390455159311671366818475253743\
1389753030284509611133126252437483447692037906219577854313866400847113429253378800186340188422282717267127\
7446941798800013407885056007128597789494624864915276343900031559421621386286951816040749879863665006223642\
7968379231155553986145475983719319768342461793856748215620959021016826718981496550260969470147755379501057\
0871078487319635714406159784123823414856661308135128869641959971454171764830607476314368599309283532459596\
5162378531656065526671721943167070640678895986993179177522970315614823290252473278902992788632382975491694\
6259458765628161322779364284162695955187356750719648087830969366509301924478254916602649933197145644072612\
2600806094874162284575537023725967774911598430990616206995450998388691312826080455593244677227883664939275\
8743180667668711574566000192267748090068027010989237133174686693964296948580295080846432971836162538329165\
8956561964586437217020653914769113466389652583356840517121256081905679039871845071608742600937525438683301\
8871827428814665985538970185165001685612286274232074507320409785004642623288826337885633979749912868706827\
6298410776568323986384042579084030981280330613384202187554063920566492475619142508184504623103906907672760\
7424143260838778880463280247926523086954407994413817029467952210485710953835097102601108914410928658908288\
0396435205077411942021811533708204269970649977576717948908847573576676954706746170190144031913873215938734\
0097237305451204681534395237703226173175970189661527784957348564658268714614824589610076901495742866042611\
9905873978506289509569105562696245065699417421049512196559480177396582198913305505623780221356128417075902\
4890105123324292379564366528567637506886724818617247421447487120021274588245933610907789638096241131835623\
9186120471694496329116157762037371412190559327054446415787317053180263168201722831458633332659136449725605\
9284059436561013892147593994240444728220732223547725918339421897287125251936596528047937761314579761653696\
9142116584642816766414682450602232199282492689140335667980313897844036321338008062813277715242334981388978\
6345671961228643084164192194436687494583669751177632032292347816867892619444681624062612125673635499520281\
7403944778558726503738384256595284958566765722626112432529877044003057744040222897210799267196210596030326\
587734997713203691232075812702157475856209;

// If the following is true then the verification was successful
hpi eq g^log;
\end{verbatim}}

\end{subappendices}
  
 
\bibliographystyle{plain}
\bibliography{papersbib}


\end{document}